\documentclass{amsart}
\usepackage{amssymb,amsmath, amsfonts, amsrefs, tikz, epsfig, float}
\usepackage{amsthm}
\usepackage{mathrsfs}
\usepackage{enumitem, indentfirst}
\usepackage[fleqn,tbtags]{mathtools}
\usepackage{color}
 \newtheorem{theorem}{Theorem}[section]
 \newtheorem{corollary}[theorem]{Corollary}

 \theoremstyle{definition}

 \theoremstyle{remark}
 \newtheorem{remark}[theorem]{Remark}
  \numberwithin{equation}{section} 

\newenvironment{enum1}{\begin{enumerate}[label=\textup{(\arabic*)}]}{\end{enumerate}}
\newenvironment{enumi}{\begin{enumerate}[label=\textup{(\roman*)}]}{\end{enumerate}}
\newenvironment{enuma}{\begin{enumerate}[label=\textup{(\alph*)}]}{\end{enumerate}}

\renewcommand{\epsilon}{\varepsilon}
\renewcommand{\phi}{\varphi}
\renewcommand{\theta}{\vartheta}
\DeclareMathOperator{\sform}{\mathfrak{s}}
\DeclareMathOperator{\tform}{\mathfrak{t}}

\DeclareMathOperator{\wform}{\mathbf{w}}

\DeclarePairedDelimiterX\sipt[2]{(}{)_{\tform}}{#1\,\delimsize\vert\,#2}
\DeclarePairedDelimiterX\sipv[2]{(}{)_{v}}{#1\,\delimsize\vert\,#2}
\DeclarePairedDelimiterX\sipw[2]{(}{)_{w}}{#1\,\delimsize\vert\,#2}

\newcommand{\abs}[1]{\lvert#1\rvert}
\newcommand{\dupN}{\mathbb{N}}

\newcommand{\seq}[1]{(#1_{n})_{n\in\dupN}}

\newcommand{\dupC}{\mathbb{C}}

\newcommand{\ran}{\operatorname{ran}}

\newcommand{\leh}{\mathscr{L}(E;\hil)}

\newcommand{\lef}{\mathscr{L}(E,F)}
\newcommand{\lefpoz}{\mathscr{L}_+(E,F)}
\newcommand{\lhaf}{\mathscr{L}(\hila,F)}

\newcommand{\D}{\mathscr{D}}

\newcommand{\hil}{\mathcal H}

\newcommand{\hila}{\hil_A}

\newcommand{\bh}{\mathscr{B}(\hil)}

\DeclarePairedDelimiterX\sip[2]{(}{)}{#1\,\delimsize\vert\,#2}
\DeclarePairedDelimiterX\siptilde[2]{(}{)_{\!_{\widetilde{A}}}}{#1\,\delimsize\vert\,#2}
\DeclarePairedDelimiterX\sipf[2]{(}{)_{f}}{#1\,\delimsize\vert\,#2}
\DeclarePairedDelimiterX\sipg[2]{(}{)_{g}}{#1\,\delimsize\vert\,#2}
\DeclarePairedDelimiterX\siptw[2]{(}{)_{\tform+\wform}}{#1\,\delimsize\vert\,#2}
\DeclarePairedDelimiterX\set[2]{\{}{\}}{#1\,:\,#2}
\DeclarePairedDelimiterX\dual[2]{\langle}{\rangle}{#1,#2}
\DeclarePairedDelimiterX\sipa[2]{(}{)_{\!_A}}{#1\,\delimsize\vert\,#2}
\DeclarePairedDelimiterX\sipc[2]{(}{)_{\!_C}}{#1\,\delimsize\vert\,#2}
\DeclarePairedDelimiterX\sipab[2]{(}{)_{\!_{A+B}}}{#1\,\delimsize\vert\,#2}
\DeclarePairedDelimiterX\sipb[2]{(}{)_{\!_B}}{#1\,\delimsize\vert\,#2}
\newcommand{\anti}[1]{\bar{#1}'}

\newcommand{\lAf}{\lambda(A,f)}

\allowdisplaybreaks
\title[Supremum and infimum of positive operators]{Operators on anti-dual pairs:\\ Supremum and infimum of positive operators}

\author[Zs. Tarcsay]{Zsigmond Tarcsay}

\address{%
Zs. Tarcsay \\ Department of Mathematics\\ Corvinus University of Budapest\\ IX. F\H ov\'am t\'er 13-15.\\ Budapest
H-1093 \\ Hungary\\ and Department of Applied Analysis  and Computational Mathematics\\ E\"otv\"os Lor\'and University\\ P\'azm\'any P\'eter s\'et\'any 1/c.\\ Budapest H-1117\\ Hungary}
\email{tarcsay.zsigmond@uni-corvinus.hu}

\author[\'A. G\"ode]{\'Abel G\"ode}

\address{%
\'A. G\"ode \\ Department of Applied Analysis  and Computational Mathematics\\ E\"otv\"os Lor\'and University\\ P\'azm\'any P\'eter s\'et\'any 1/c.\\ Budapest H-1117\\ Hungary}
\email{godeabel@student.elte.hu}


\dedicatory{Dedicated to Professor Zolt\'an Sebesty\'en on the occasion of his 80th birthday}

\begin{document}

\begin{abstract} Our purpose in this note is to investigate the order properties of positive operators from a locally convex space into its conjugate dual. We introduce a natural generalization of the Busch-Gudder strength function and we prove  Kadison's anti-lattice theorem and Ando's result on the infimum of positive operators in that context.  
\end{abstract}

\maketitle

\section{Introduction}
Let $\hil$ be a complex (possibly infinite dimensional) Hilbert space. It is well known that the set of bounded self-adjoint operators $B_{sa}(\hil)$ is a partially ordered set with respect to the most natural L\"owner order
\begin{equation*}
    A\leq B\quad \iff\quad \sip{Ax}{x}\leq \sip{Bx}{x}\qquad (\forall x\in\hil).
\end{equation*}
That is to say, `$\leq $' is a reflexive, transitive and anti-symmetric relation. It is however proved by Kadison \cite{kadison1951}  that the partially ordered set $(B_{sa}(\hil),\leq )$ is  as far from being a lattice as possible. In fact, the supremum $A\vee B$ of the self-adjoint operators $A$ and $B$ can exist only in the trivial case when they are comparable.  The same is true for the question of the infimum, as we have 
\begin{equation*}
    A\wedge B=-((-A)\vee(-B)).
\end{equation*}
If we examine the same questions in the cone $B_+(\hil)$ of positive operators, the answer does not change in the case of the supremum.
However, it is not difficult to see that the infimum of two orthogonal projections  $P,Q$ with disjoint ranges is $P\wedge Q=0$ in $B_+(\hil)$. This simple example suggests that the infimum problem considered over the cone of positive operators is a more sophisticated problem than that over $B_{sa}(\hil)$.

The latter issue was completely solved by T. Ando \cite{ando1999problem}. He  showed that the infimum problem is closely related to the so-called Lebesgue-type decomposition of positive operators, the theory of which was also developed by him \cite{ando1976lebesgue}. Namely, given two positive operators $A,B$ on $\hil$, there exists two positive operators $B_a,B_s$ on $\hil$, where $B_a$ is absolutely continuous to $A$, whilst $B_s$ is singular to $A$,  such that $B_a+B_s=B$ (see \cite{ando1976lebesgue}). It is also proved by Ando that such a decomposition is not unique in general, but there is a distinguished positive operator (denoted by $[B]A\coloneqq B_a$) which satisfies the decomposition and which is the largest among those operators $C\geq 0$ such that $C\leq A$ and $C\ll B$. (We refer the reader to \cite{ando1976lebesgue} and \cite{tarcsay2013lebesgue} for the details.) This maximal operator $[B]A$ is called the \textit{$B$-absolutely continuous part} of $A$.

From the maximality property of the absolutely continuous parts it is easy to check that if $[B]A$ and $[A]B$ are  comparable, then the infimum of $A$ and $B$ exists, namely,
\begin{equation*}
    A\wedge B=\max\{[A]B,[B]A\}.
\end{equation*}
The main result of Ando's paper \cite{ando1999problem} states that  the reverse of this is also true: $A\wedge B$ exists if and only if the corresponding absolutely continuous parts are comparable.

The  aim of this note is to investigate the supremum and infimum problem in a much more general context, namely, among the positive cone of operators on so called anti-dual pairs. By anti-dual pair we mean a pair $(E,F)$  of complex vector spaces  which are connected by a so-called anti-duality function $\dual{\cdot}{\cdot }:F\times E\to\dupC$. The latter differs from the well-known duality function (see \cite{Schaefer}*{Chapter IV}) only in that it is conjugate linear in its second variable. The notion of anti-dual pairs allows us to define positivity of linear operators of  type $A:E\to F$ in a way that is formally identical to the positivity of operators over Hilbert spaces (cf. \cite{TARCSAY2020Lebesgue} and \cite{TarcsayMathNach}). The set  $\lefpoz$ of positive operators from $E$ to $F$ is then a partially ordered set via the ordering
\begin{equation*}
    A\leq B\qquad\overset{def}{\Longleftrightarrow}\qquad \dual{Ax}{x}\leq \dual{Bx}{x}\qquad (\forall x\in E).
\end{equation*}
In this article we are going to present the corresponding analogous versions to Kadison's and Ando's theorems in the ordered set $(\lefpoz,\leq )$. In the case of supremum, the key is a corresponding generalization of the  Busch-Gudder strength function \cite{busch1999effects} to positive operators in $\lefpoz$. This will enable us to provide a numerical characterization of the ordering. The solution to the infimum problem is based on the  Lebesgue decomposition theory for operators on anti-dual pairs developed in \cite{TARCSAY2020Lebesgue}.

\section{Preliminaries}

Throughout the paper, let   $E$ and $F$ be complex vector spaces which are intertwined via a function $$\dual\cdot\cdot:F\times E\to\dupC,$$ which is linear in its first argument and conjugate linear in its second one, so that it separates the points of $E$ and $F$. We shall refer to $\dual\cdot\cdot$ as \emph{anti-duality} function and the triple  $(E,F,\dual\cdot\cdot)$ will be called an \emph{anti-dual pair}. and shortly denoted by $\dual FE$. 

The prototype of anti-dual pairs is the triple $(\hil,\hil,\sip{\cdot}{\cdot})$ , where $\hil$ is a Hilbert space and $\sip{\cdot}{\cdot}$ is a scalar product over $\hil$. In that case,  of course, the concept of symmetric and positive operators coincides with the usual concept of those in functional analysis. 

However, the most general example of anti-dual pairs is $(X,\anti{X},\dual{\cdot}{\cdot})$, where $X$ is a locally convex Hausdorff space, $\anti{X}$ denotes the set of continuous and conjugate linear functionals on $X$, and $\dual{\cdot}{\cdot}$ is the  evaluation 
\begin{equation*}
    \dual{f}{x}\coloneqq f(x)\qquad (x\in X, f\in\anti X). 
\end{equation*}

If $(E,F)$ is an anti-dual pair, then we may consider the corresponding weak topologies $w\coloneqq w(E,F)$ on $E$, and $w^*\coloneqq w(F,E)$ on $F$, respectively. Both $(E,w)$ and $(F,w^*)$ are locally convex Hausdorff spaces such that 
\begin{equation}\label{E:E'=F}
    \anti E=F,\qquad F'=E.
\end{equation} 
This fact and \eqref{E:E'=F} enable us to define the adjoint (that is, the topological transpose) of a weakly continuous operator. Let $\dual{F_1}{E_1}$ and $\dual{F_2}{E_2}$ be anti-dual pairs and  $T:E_1\to F_2$ a weakly continuous linear operator, then the (necessarily weakly continuous) linear operator $T^*:E_2\to F_1$ satisfying \begin{equation*}
    \dual{Tx_1}{x_2}_2=\overline{\dual{T^*x_2}{x_1}_1},\qquad  x_1\in E_1,x_2\in E_2
\end{equation*} 
is called the adjoint of $T$. In the following, we will use two particularly important special cases of this: on the one hand, the adjoint $T^*$ of a weakly continuous operator $T:E\to F$ is also an operator of type $E\to F$. On the other hand, if $\hil$ is a Hilbert space, then the adjoint of a weakly continuous operator $T:E\to \hil$ acts as and operator $T^*:\hil\to F$, so that their composition   $T^*T:E\to F$ satisfies
\begin{equation}\label{E:TadjTxx}
    \dual{T^*Tx}{x}=\sip{Tx}{Tx} \qquad (x\in E).
\end{equation}
A linear operator $S:E\to F$ is called \textit{symmetric} if 
\begin{equation*}
    \dual{Sx}{y}=\overline{\dual{Sy}{x}}\qquad (x,y\in E),
\end{equation*}
and positive, if 
\begin{equation*}
    \dual{Sx}{x}\geq 0\qquad(x\in E).
\end{equation*}
It can be readily checked that  every positive operator is symmetric, and every symmetric operator $S$ is weakly continuous (see \cite{TARCSAY2020Lebesgue}). We shall denote the set of weakly continuous operators $T:E\to F$ by $\lef$, while we write $\lefpoz$ for the set of positive operators $A:E\to F$.  

According to \eqref{E:TadjTxx}, $T^*T\in\lef$ is a positive operator if $T\in\leh$. However, if we assume $F$ to be $w^*$-sequentially complete, then we can also state the converse: every positive operator $A\in\lefpoz$ can be written in the form $A=T^*T$ for some auxiliary Hilbert space $\hil$ and $T\in\leh$ (cf. \cite{TARCSAY2020Lebesgue}).  Recall that the anti-dual pair $(X,\anti X)$ is $w^*$-sequentially complete if $X$ is a barreled space (e.g., a Banach space). Also,  $(X,\bar X^*)$ is $w^*$-sequentially complete with $X$ being an arbitrary vector space and $\bar X^*$ its algebraic anti-dual space. 

Let us now recall briefly the Hilbert space factorization method of a positive operator $A$. Let $(E,F)$ be a $w^*$-sequentially anti-dual pair and let $A\in\lefpoz$. Then
\begin{equation*}
    \sipa{Ax}{Ay}\coloneqq \dual{Ax}{y}\qquad (x,y\in E).
\end{equation*}
defines an inner product on the range space $\ran A$ under which it becomes a pre-Hilbert space. We shall denote the completion by $\hila$.  The canonical embedding operator
\begin{equation}\label{E:J_A}
        J^{}_A(Ax)=Ax\qquad (x\in E)
    \end{equation}
of $\ran A\subseteq \hila$ into $F$ is weakly continuous. By $w^*$-sequentially completeness, it uniquely extends to a weakly continuous operator $J_A\in \lhaf$. It can be checked that the adjoint operator $J_A^*\in\mathscr{L}(E,\hila)$  satisfies 
\begin{equation}\label{E:J*}
        J_A^*x=Ax\in\hila \qquad (x\in E),
\end{equation}
that yields the fallowing factorization of $A$:
\begin{equation}\label{E:JAJA}
    A=J_A^{}J_A^*.
\end{equation}

In order to discuss  the infimum problem of positive operators, we will need some concepts and results related to the  Lebesgue decomposition theory of positive operators. Hence, for sake of the reader, we briefly summerize the content of \cite{TARCSAY2020Lebesgue} which is essential for the understanding of this article. The details can be found in the mentioned paper. 

We say that the positive operator $B\in\lefpoz$ is  absolutely continuous with respect to another positive operator $A\in\lefpoz$ (in notation, $B\ll A$) if, for every sequence $\seq x$ of $E$, $\dual{Ax_n}{x_n}\to0$ and $\dual{B(x_n-x_m)}{x_n-x_m}\to0$ imply $\dual{Bx_n}{x_n}\to0$.  $A$ and $B$ are said to be mutually singular (in notation, $A\perp B$) if the only positive operator $C\in\lefpoz$ for which $C\leq A$ and $C\leq B$ are satisfied is $C=0$. It is easy to check that $B\leq A$ imples $B\ll A$, however the converse is apparently not true (cf. \cite{TARCSAY2020Lebesgue}*{Theorem 5.1}).

In \cite{TARCSAY2020Lebesgue}*{Theorem 3.3} it was proved that every positive operator $B$ has a Lebesgue-type decomposition with respect to any other positive operator $A$. More precisely, there exists a pair $B_a$ and $B_s$ of positive operators such that 
\begin{equation}\label{E:LebDec}
    B=B_a+B_s,\qquad B_a\ll A,\qquad B_s\perp A.
\end{equation}
Such a decomposition is not unique in general (cf \cite{TARCSAY2020Lebesgue}*{Theorem 7.2}, however,  there is a unique operator $B_a\in\lefpoz$  satisfying \eqref{E:LebDec} which is maximal among those positive operators $C$ such that $C\leq B$ and $C\ll A$. We shall call this uniquely determined operator $B_a$ the $A$-\textit{absolutely continuous part} of $B$, and adopting Ando's \cite{ando1976lebesgue} notation, we shall write $[A]B\coloneqq B_a$ for it. 


\section{The strength of a positive operator}
Throughout the section let  $(E,F)$ be a  $w^*$-sequentially complete anti-dual pair. Let us introduce the partial order on  $\lefpoz$ by 
\begin{equation*}
    A\leq B \qquad \iff \qquad\dual{Ax}{x}\leq \dual{Bx}{x}\qquad (x\in E).
\end{equation*}
In what follows, inspired by  the paper \cite{busch1999effects} of Busch and Gudder, we associate a function to each positive operator $A$ (called the strength function) which can be used to characterize that ordering. 

Let  $f\in F$ be a non-zero vector and set
\begin{equation*}
    (f\otimes f)(x)\coloneqq \overline{\dual fx}\cdot f\qquad(x\in E),
\end{equation*}
so that $f\otimes f\in\lefpoz$ is a rank one positive operator with range space $\dupC\cdot f$. For a given positive operator $A\in\lefpoz$ we set 
 \begin{equation*}
    \lAf\coloneqq \sup \set{t\geq0}{t\cdot f\otimes f\leq  A}.
\end{equation*}  
Following the terminology of \cite{busch1999effects}, the nonnegative (finite) number $\lAf$ will be called the strength of $A$ along the ray $f$, whilst the  function  
\begin{equation*}
\lambda(A,\cdot): F\setminus\{0\}\to[0,+\infty)    
\end{equation*}
is the \textit{strength function} of $A$. To see that  $\lAf$ is always finite, consider an $x\in E$  such that $\dual{f}{x}\neq 0$. Then $\lAf=+\infty$ would result  equality $\dual{Ax}{x}=+\infty$, which is impossible.  

At this point, we note that in \cite{busch1999effects}, the strength function was only defined along vectors from the unit sphere of  a Hilbert space. In that case, the strength function has the uniform upper bound $\|A\|$.  By the above interpretation of the strength function, $\lambda(A,\cdot)$ will not be bounded, but this fact is not of considerable importance for the applications.

In our first result, we examine along which 'rays' $f$ the strength function takes a positive value. The factorization \eqref{E:JAJA} of the positive operator $A$  and the auxiliary Hilbert space $\hila$ will play an important role in this. 
\begin{theorem}\label{T:3.1}
Let $A\in\lefpoz$ and $f\in F$, $f\neq 0$. The following are equivalent:
\begin{enumerate}[label=\textup{(\roman*)}]
    \item $\lAf>0$,
    \item there is a constant $m\geq0$ such that $\abs{\dual fx}^2\leq m\cdot\dual{Ax}{x}$, $(x\in E),$ 
    \item there is a (unique) $\xi_f\in\hila$ such that $J_A\xi_f=f$.
\end{enumerate}
In any case, we have 
\begin{equation}
    \lAf=\frac{1}{\|\xi_f\|_A^2}=\frac{1}{m_f},
\end{equation}
where $m_f>0$ is the smallest constant $m$ that satisfies (ii).
\end{theorem}
\begin{proof}
(i)$\Rightarrow$(ii): Fix a real number $0<t<\lAf$, then 
\begin{equation*}
    \dual{(f\otimes f)x}{x}\leq \frac1t\cdot\dual{Ax}{x}\qquad x\in E, 
\end{equation*}
that implies (ii). 

(ii)$\Rightarrow$(iii): Inequality (ii) expresses just that  
\begin{equation*}
    \phi(Ax)\coloneqq \dual fx,\qquad x\in E,
\end{equation*}
defines a continuous conjugate linear functional from $\ran A\subseteq \hila$ to $\dupC$, namely $\|\phi\|^2\leq m$. Denote by $\xi_f\in\hila$ the corresponding Riesz representing vector. Then 
\begin{equation*}
    \dual{f}{x}=\sipa{\xi_f}{Ax}=\sipa{\xi_f}{J_A^*x}=\dual{J_A\xi_f}{x},\qquad x\in E,
\end{equation*}
and hence $f=J_A\xi_f$.

(iii)$\Rightarrow$(i): Suppose that $J_A\xi_f=f$ for some $\xi_f\in\hila$. Then 
\begin{equation*}
    \abs{\dual fx}^2=\abs{\sipa{\xi_f}{Ax}}^2\leq \|\xi_f\|^2_A\dual{Ax}{x},
\end{equation*}
hence $\lambda f\otimes f\leq A$ with $\lambda=\|\xi_f\|_A^{-2}$.
\end{proof}
As a corollary we retrieve \cite{busch1999effects}*{Theorem 3}. The proof presented here  is significantly shorter and simpler.
\begin{corollary}
Let $\hil$ be a Hilbert space, $A\in\bh$ and $f\in \hil$, then $\lambda(A,f)>0$ if and only if $f\in\ran A^{1/2}$. In that case,  
\begin{equation}\label{E:BGthm3}
    \lambda(A,f)=\frac{1}{\|A^{-1/2}f\|^2},
\end{equation}
where $A^{-1/2}$ denotes the (possibly unbounded) inverse of $A^{1/2}$ restricted to $\ker A^{\perp}$.
\end{corollary}
\begin{proof}
    The first part of the statement is clear because $A^{1/2}A^{1/2}=A=J_A^{}J_A^*$ implies $\ran A^{1/2}=\ran J_A$. In order to prove \eqref{E:BGthm3}, consider the linear functional 
    \begin{equation*}
        \phi:\ran A^{1/2}\to \dupC,\qquad \phi(A^{1/2}x)\coloneqq \sip xf\qquad (x\in \hil)
    \end{equation*}
    which is well-defined and bounded with norm $\|\phi\|^2=\lambda(A,f)^{-1}$. By the Riesz representation theorem, there is a unique vector $\zeta\in\overline{\ran A^{1/2}}=\ker A^{\perp}$ such that  
    \begin{equation*}
        \sip xf=\sip{A^{1/2}x}{\zeta}=\sip{x}{A^{1/2}\zeta}\qquad (x\in \hil).
    \end{equation*}
    Hence $A^{1/2}\zeta=f$ and thus $\zeta=A^{-1/2}f$. Furthermore, $\|\zeta\|^2=\|\phi\|^2=\lambda(A,f)^{-1}$ that completes the proof.
\end{proof}
Below we establish the most useful property of the  strength function for our purposes. Namely,  it can be used to characterize the ordering.
\begin{theorem}
Let $A,B\in\lefpoz$ be positive operators then the following assertions are equivalent:
\begin{enumerate}[label=\textup{(\roman*)}]
    \item $A\leq B$,
    \item $\lAf\leq \lambda (B,f)$ for every $0\neq f\in F$.
\end{enumerate}
\end{theorem}
\begin{proof}
(i)$\Rightarrow$(ii): If $A\leq B$ and $\lambda \cdot f\otimes f\leq A$ for some $\lambda $ then also $\lambda \cdot f\otimes f\leq B$ hence $\lambda (A,f)\leq \lambda(B,f)$.

(ii)$\Rightarrow$(i): Assume $\lambda(A,\cdot)\leq \lambda(B,\cdot)$ everywhere on $F\setminus\{0\}$. By contradiction suppose $\dual{Ax}{x}>\dual{Bx}{x}$ for some $x$ and set $f\coloneqq \dual{Ax}{x}^{-1}\cdot Ax$. For every $y\in E$, 
\begin{equation*}
    \dual{(f\otimes f)(y)}{y}=\frac{\abs{\dual{Ax}{y}}^2 }{\dual{Ax}{x}}\leq \dual{Ay}{y}
\end{equation*}
by the Cauchy-Schwarz inequality $\abs{\dual{Ax}{y}}^2\leq \dual{Ax}{x}\dual{Ay}{y}$. Hence $\lambda(A,f)\geq 1$ and thus, by assumption,
\begin{equation*}
    \dual{Bx}{x}\geq \dual{(f\otimes f)(x)}{x}=\dual{Ax}{x},
\end{equation*}
which leads to a contradiction.
\end{proof}
\begin{corollary}
Two positive operators are identical if and only if so are their strength functions. 
\end{corollary}

\begin{remark}
From the very definition of the strength function it is readily check the following elementary properties: 

Let $A,B\geq 0$ and $0\neq f\in F$, then
\begin{enum1}
    \item $\lambda(\alpha A,f )=\alpha\lambda(A,f )$ for every $\alpha\geq 0$,
    \item $\lambda(\alpha A,f)+\lambda ((1-\alpha)B,f)\geq \alpha\lambda(A,f)+(1-\alpha)\lambda (B,f)$ for every $\alpha\in[0,1]$,
    \item $\lambda(A+B,f)\geq \lambda(A,f)+\lambda(B,f)$. 
\end{enum1}
\end{remark}
\section{Supremum of positive operators}
In this section, we are going to prove the  Kadison's anti-lattice theorem in the setting of positive operators on an anti-dual pair. Namely, we show that the supremum of two positive operators $A$ and $B$  exists if and only if they are comparable (that is, $A\leq B$ or $B\leq A$). 

\begin{theorem}\label{T:4.1}
Let $(E,F)$ be a $w^*$-sequentially complete anti-dual pair and let $A, B\in\lefpoz$ positive operators. Then the following two statements are equivalent:
\begin{enumerate}[label=\textup{(\roman*)}]
    \item the supremum $A\vee B$ exists,
    \item $A$ and $B$ are comparable.
\end{enumerate}
\end{theorem}
\begin{proof}
The comparability of $A$ and $B$ trivially implies the existence of $A\vee B$. For the opposite direction we suppose that $T\in\lefpoz\setminus\{A, B\}$ such that $T\ge A$ and $T\ge B$ and show that there exists $S\in\lefpoz$ such that $S\ge A$ and $S\ge B$, but $S$ is not comparable with $T$.

In the first case we suppose that the intersection of $\ran(J_{T-A})$ and $\ran(J_{T-B})$ contains a non-zero vector $e\in\ran(J_{T-A})\cap\ran(J_{T-B})$. By Theorem \ref{T:3.1} there exists $\lambda>0$ such that $\lambda\cdot e\otimes e\le T-A$ and $\lambda\cdot e\otimes e\le T-B$. Let $f\in F\setminus\mathbb{C}\cdot e$, then with
\begin{equation*}
    S\coloneqq T-\lambda\cdot e\otimes e+f\otimes f
\end{equation*}
we have $S\ge A$ and $S\ge B$, but $S$ is apparently not comparable with $T$.

In the second case we suppose that $\ran(J_{T-A})\cap\ran(J_{T-B})=\{0\}$. By Theorem \ref{T:3.1} there exist $e\in \ran(J_{T-A})$ and $f\in\ran (J_{T-B})$ such that $e\otimes e\le T-A$ and $f\otimes f\le T-B$. Let us define
\begin{equation*}
    S_0\coloneqq e\otimes e+2e\otimes f+2f\otimes e+f\otimes f
\end{equation*}
and
\begin{equation*}
    S:=T+\frac{1}{3}S_0.
\end{equation*}
Clearly, $S_0$ and $S$ are symmetric operators. We claim that  $S\ge A$ and $S\ge B$. Indeed, we have
\begin{equation*}
    S_0+3e\otimes e=(2e+f)\otimes(2e+f)\ge 0,
\end{equation*}
which implies
\begin{equation*}
    -\frac{1}{3}S_0\le e\otimes e\le T-A,
\end{equation*}
therefore
\begin{equation*}
    \frac{1}{3}S_0\ge A-T.
\end{equation*}
A similar argument shows that $  \frac{1}{3}S_0\ge B-T$.
As a consequence we see that
\begin{equation*}
    S=T+\frac{1}{3}S_0\ge T+(B-T)=B,
\end{equation*}
and similarly, $S\geq A$. However, the symmetric operator $S_0$ is not positive. For let $x\in E$ be  such that $\langle e,x\rangle=1$, and $\langle f,x\rangle=-1$, then $\langle S_0x,x\rangle=-2$. Consequently,  $S$ and $T$ are not comparable, which proves the theorem.
\end{proof}

\section{Infimum of positive operators}
Let $A,B\in\lefpoz$ be positive operators and consider the corresponding $A$-, and $B$-absolute continuous parts $[A]B$ and $[B]A$, arising from the corresponding Lebesgue decompositions in the sense of \eqref{E:LebDec}. Thus,
\begin{equation}\label{E:[A]B}
    [A]B=\max\set{C\in\lefpoz}{C\leq B, C\ll A},
\end{equation}
and 
\begin{equation}\label{E:[B]A}
     [B]A=\max\set{C\in\lefpoz}{C\leq A, C\ll B},
\end{equation}
see \cite{TARCSAY2020Lebesgue}*{Theorem 3.3}. Assume for a moment that $[A]B\leq [B]A$, then clearly, $[A]B\leq A$, $[A]B\leq B$ and  for every $C\in\lefpoz$ such that $C\leq A$,$C\leq B$ we also have $C\leq [A]B$. This in turn means that $[A]B$ is the infimum of $A$ and $B$ in the cone $\lefpoz$. The situation is essentially same if we assume that $[B]A\leq [A]B$.

Ando's result \cite{ando1999problem}*{Theorem 6} states that, in the context of positive operators on Hilbert spaces, the infimum of two operators only in one of the above two cases exists. That is, if $A\wedge B$ exists in the cone $B_+(\hil)$, then $[A]B$ and $[B]A$ are comparable (and the smaller one is the infimum). 

Our aim in the present section is to establish the anti-dual pair analogue of Ando's mentioned result:
\begin{theorem}\label{T:5.1}
Let $(E,F)$ be a $w^*$-sequentially complete anti-dual pair and let $A, B\in\lefpoz$ positive operators. Then the following two statements are equivalent:
\begin{enumi}
 \item the infimum $A\wedge B$ exists in $\lefpoz$, 
 \item the corresponding absolute continuous parts $[A]B$ and $[B]A$ are comparable.
\end{enumi}
In any case,  $A\wedge B=\min\{[A]B,[B]A\}$.
\end{theorem}
\begin{proof}
We only prove the non-trivial implication (i)$\Rightarrow$(ii). First of all we remark that $A\wedge B$ exists if and only if $[A]B\wedge [B]A$ exists, and in that case these two operators coincide. This is easily obtained from the maximality properties \eqref{E:[A]B} and \eqref{E:[B]A}. Recall also that $[A]B$ and $[B]A$ are mutually absolutely continuous according to \cite{TARCSAY2020Lebesgue}*{Theorem 3.6}, so we may assume without loss of generality that $A\ll B$ and $B<\ll A$. 

Consider now the auxiliary Hilbert space $\hil\coloneqq \hil^{}_{A+B}$ associated with the sum $A+B$, and denote by $\sip{\cdot}{\cdot}$ its inner product and by $J\coloneqq J_{A+B}^{}$ the corresponding canonical embedding of $\hil$ into $F$. Since the nonnegative forms 
\begin{equation*}
    ((A+B)x,(A+B)y)\mapsto \dual{Ax}{y},\qquad  ((A+B)x,(A+B)y)\mapsto  \dual{Bx}{y},
\end{equation*}
defined on the dense subspace $\ran (A+B)\subseteq \hil$ are obviously bounded by norm $\leq 1$, the Lax-Milgram lemma provides us two positive operators $\widetilde A,\widetilde B\in\mathscr B(\hil)$ with $\|\widetilde A\|,\|\widetilde B\|\leq 1$, which satisfy
\begin{equation*}
   \sip{\widetilde A(A+B)x}{(A+B)y}=\dual{Ax}{y},\qquad 
   \sip{\widetilde B(A+B)x}{(A+B)y}=\dual{Bx}{y}. 
\end{equation*}
Using the canonical property 
\begin{equation*}
    J^*x=(A+B)x,\qquad x\in E
\end{equation*}
we obtain that
\begin{equation*}
     A=J\widetilde AJ^*\qquad\mbox{and}\qquad B=J\widetilde{B} J^*.
\end{equation*}
Observe also that 
\begin{equation*}
    \sip{(\widetilde A+\widetilde B)J^*x}{J^*y}=\dual{(A+B)x}{y}=\sip{J^*x}{J^*y},\qquad (x,y\in E),
\end{equation*}
whence 
\begin{equation*}
    \widetilde A+\widetilde B=I,
\end{equation*}
where $I$ stands for the identity operator on $\hil$. Note also that 
\begin{equation}\label{E:ker}
    \ker \widetilde A=\ker \widetilde B =\{0\},
\end{equation}
because $A\ll B$ and $B\ll A$ (see \cite{TARCSAY2020Lebesgue}*{Lemma 5.2}).

Denote by $C\coloneqq A\wedge B$ and let $\widetilde C\in\bh$ be the positive operator such that $C=J\widetilde C J^*$. We claim that $\widetilde C=\widetilde A\wedge \widetilde B$ in $\mathscr{B}_+(\hil)$. Indeed, it is clear on the one hand that $\widetilde  C\leq \widetilde A$ and $\widetilde C\leq \widetilde{B}$. On the other hand, consider a positive operator $\widetilde D\in B_+(\hil)$  such that $\widetilde D\leq \widetilde A, \widetilde{B}$. It is readily seen  that  $D\coloneqq J\widetilde{D}J^*\in \lefpoz$ satisfies $D\leq A,B$, whence $D\leq C$ which implies $\widetilde D\leq \widetilde C$.

Let $E$ be the spectral measure of $\widetilde A$. Following the reasoning of \cite{ando1999problem}*{Lemma 1} it follows that 
\begin{equation}\label{E:tn1-t}
    \widetilde A\wedge \widetilde B=\widetilde A\wedge (I-\widetilde A)=\int_{0}^{1} \min\{t,(1-t)\}\,dE(t).
\end{equation}
The statement of the theorem will be proved if we show that either  $\widetilde A\wedge \widetilde B=\widetilde A$ or $\widetilde A\wedge \widetilde B=\widetilde{B}$. By \eqref{E:tn1-t} this is equivalent to prove that 
\begin{equation}\label{E:spectrum}
    \sigma(\widetilde A)\subseteq [0,\tfrac12]\qquad\mbox{or}\qquad \sigma(\widetilde A)\subseteq [\tfrac12,1]
\end{equation}
holds on the spectrum of $\widetilde A$.
Suppose that is not the case. By \eqref{E:ker} we have $E(\{0\})=E(\{1\})=0$, hence there exists some $\epsilon>0$ such that none of the spectral projections 
\begin{equation*}
    P_1\coloneqq E([\tfrac12+3\epsilon,1-3\epsilon]) \qquad \mbox{and}\qquad P_2\coloneqq E([3\epsilon,\tfrac12-3\epsilon])
\end{equation*}
is  zero. Suppose  that $0<\dim P_2\leq \dim P_1$ and take a partial isometry $V\in\bh$ such that $[\ker V]^\perp\subseteq \ran P_2$ and $\ran V\subseteq \ran P_1$. Then, following Ando's treatment used in the proof of \cite{ando1999problem}*{Theorem 1}, one can prove that 
\begin{equation*}
    \widetilde D\coloneqq (\widetilde A-\epsilon)\cdot P_2+(\widetilde B-\epsilon)\cdot P_1+\sqrt{2}\epsilon (VP_2+P_2V^*)
\end{equation*}
satisfies $0\leq \widetilde D\leq \widetilde A,\widetilde B$, but  $\widetilde D$ is not comparable with $\widetilde C$. This, however contradicts the fact that $\widetilde C=\widetilde A\wedge \widetilde B$, which proves the theorem.
\end{proof}



  As a direct application of  Theorems \ref{T:4.1} and \ref{T:5.1}, we obtain the following result regarding the supremum and infimum of non-negative forms over a given vector space. With the latter statement we retrieve \cite{titkos2012ando}*{Theorem 3}.
\begin{corollary}
Let $\sform$ and $\tform$ be nonnegative forms on a complex vector space $\D$. 
\begin{enuma}
    \item The supremum $\tform\vee \sform$ exists if and only if $\tform\leq \sform$ or $\sform\leq \tform$,
    \item The infimum $\tform\wedge\sform$ exists if and only if $[\tform]\sform\leq [\sform]\tform$ or $[\sform]\tform\leq [\tform]\sform$.
\end{enuma}
\end{corollary}
\begin{proof}
Denote by $\bar \D^*$ the conjugate algebraic dual of $\D$ (that is, $\bar\D^*$ consists of all anti-linear forms $f:\D\to\dupC$). Then $(\D,\bar \D^*)$, endowed with the natural anti-duality $\dual\cdot\cdot$, forms a $w^*$-sequentially continuous anti-dual pair. Note that every nonnegative form $\tform$ on $\D$ induces a positive operator $T:\D\to\bar \D^*$ by the correspondence
\begin{equation*}
    \dual{Tx}{y}\coloneqq \tform(x,y),\qquad (x,y\in\D).
\end{equation*}
 Conversely, every positive operator  $T:\D\to\bar \D^*$  defines a nonnegative form in the most natural way. The correspondence $\tform\mapsto T$ is an order preserving bijection between nonnnegative forms and positive operators, so that both statements (a) and (b) of the theorem follows from Theorems \ref{T:4.1} and \ref{T:5.1}, respectively. 
\end{proof}
\bibliography{references.bib}

\begin{bibdiv}
\begin{biblist}

\bib{ando1976lebesgue}{article}{
      author={Ando, T.},
       title={Lebesgue-type decomposition of positive operators},
        date={1976},
     journal={Acta Sci. Math.(Szeged)},
      volume={38},
      number={3-4},
       pages={253\ndash 260},
}

\bib{ando1999problem}{article}{
      author={Ando, T.},
       title={Problem of infimum in the positive cone},
        date={1999},
     journal={Analytic and geometric inequalities and applications},
       pages={1\ndash 12},
}

\bib{busch1999effects}{article}{
      author={Busch, P.},
      author={Gudder, S.P.},
       title={Effects as functions on projective {H}ilbert space},
        date={1999},
     journal={Letters in Mathematical Physics},
      volume={47},
       pages={329\ndash 337},
}

\bib{kadison1951}{article}{
      author={Kadison, R.},
       title={Order properties of bounded self-adjoint operators},
        date={1951},
     journal={Proceedings of the American Mathematical Society},
      volume={2},
      number={3},
       pages={505\ndash 510},
}

\bib{Schaefer}{book}{
      author={Schaefer, H.~H.},
       title={Topological vector spaces},
   publisher={Springer-Verlag, New York-Berlin},
        date={1971},
}

\bib{tarcsay2013lebesgue}{article}{
      author={Tarcsay, Zs.},
       title={Lebesgue-type decomposition of positive operators},
        date={2013},
     journal={Positivity},
      volume={17},
       pages={803\ndash 817},
}

\bib{TARCSAY2020Lebesgue}{article}{
      author={Tarcsay, Zs.},
       title={{Operators on anti-dual pairs: Lebesgue decomposition of positive
  operators}},
        date={2020},
        ISSN={0022-247X},
     journal={Journal of Mathematical Analysis and Applications},
      volume={484},
      number={2},
       pages={123753},
  url={http://www.sciencedirect.com/science/article/pii/S0022247X19310212},
}

\bib{TarcsayMathNach}{article}{
      author={Tarcsay, Zs.},
      author={Titkos, T.},
       title={{Operators on anti-dual pairs: Generalized Krein-von Neumann
  extension}},
        date={2021},
     journal={Mathematische Nachrichten},
      volume={294},
      number={9},
       pages={1821\ndash 1838},
}

\bib{titkos2012ando}{article}{
      author={Titkos, T.},
       title={Ando's theorem for nonnegative forms},
        date={2012},
     journal={Positivity},
      volume={16},
      number={4},
       pages={619\ndash 626},
}

\end{biblist}
\end{bibdiv}
\end{document}